\documentclass[11pt,a4paper]{article}
\usepackage[latin1]{inputenc}
\usepackage{amsfonts}
\usepackage{amscd}
\usepackage{amsmath}
\usepackage{geometry}
\usepackage{color}
\usepackage{theorem}
\usepackage{epsfig}
\usepackage{graphics}
\usepackage{adjustbox}
\usepackage{mathrsfs}
\usepackage{dsfont}
\usepackage{capt-of}
\usepackage{pst-plot, pstricks}
\usepackage{fancybox,amssymb}
\usepackage[english]{babel}
\usepackage{comment}
\newcommand{\Real}{\mathbb{R}}
\newcommand{\Natural}{\mathbb{N}}

\newcommand{\bq}{\begin{eqnarray*}}
\newcommand{\eq}{\end{eqnarray*}}

\makeatletter\def\blfootnote{\xdef\@thefnmark{}\@footnotetext}\makeatother
\theoremstyle{break}
\newtheorem{defn}{Definition}[section]
\newtheorem{theorem}[defn]{Theorem}

\newtheorem{lemma}[defn]{Lemma}
\newtheorem{exmp}[defn]{Example}
\newtheorem{remark}{Remark}[section]

\makeatletter
\newenvironment{proof}{\noindent{\textit{Proof:}}}{%
\unskip\nobreak\hfil\penalty50\hskip1em\null\nobreak
$\Box$
\parfillskip=\z@\finalhyphendemerits=0\endgraf\bigskip} 
\let\oldendexmp\endexmp
\def\endexmp{\unskip\nobreak\hfil\penalty50\hskip1em\null\nobreak\hfil%
$\blacksquare$\parfillskip=\z@\finalhyphendemerits=0\endgraf\oldendexmp}
\makeatother

\title{{\Large Integral equations, quasi-Monte Carlo methods and risk modelling}\\
{\large Dedicated to the $80$th anniversary of Ian Sloan}
}
\author{\normalsize M. Preischl, S. Thonhauser and R.F. Tichy\thanks{The authors
are supported by the Austrian Science Fund (FWF) Project F5510 (part of the Special Research
Program (SFB) \textquotedblleft Quasi-Monte Carlo Methods: Theory and Applications\textquotedblright).
}}
\date{}
\newcommand{\Addresses}{{
\bigskip
\footnotesize
\textsc{Institute of Analysis and Number Theory, Graz University of Technology,
Steyrergasse 30/II, 8010 Graz, Austria}\par\nopagebreak
\textit{E-mail addresses}: \texttt{preischl@math.tugraz.at, stefan.thonhauser@math.tugraz.at, tichy@tugraz.at}
}}

\begin{document}
\maketitle
\abstract{
We survey a QMC approach to integral equations and develop some new applications to risk modeling.
In particular, a rigorous error bound derived from Koksma-Hlawka type inequalities is achieved for certain expectations related to the probability of ruin
in Markovian models. The method is based on a new concept of isotropic discrepancy and its applications to numerical integration.
The theoretical results are complemented by numerical examples and computations.
}
\section{Introduction}
During the last two decades quasi-Monte-Carlo methods
(QMC-methods) have been applied to various problems in numerical
analysis, statistical modeling and mathematical finance. In this
paper we will give a brief survey on some of these developments and present new applications to more refined risk models
involving discontinuous processes.
Let us start with Fredholm integral equations of the second kind:

\begin{align}f(\textbf{x})=g(\textbf{x})+\int_{[0,1]^{s}}K(\textbf{x},\textbf{y})f(\textbf{y})d\textbf{y},
\end{align}
where the kernel is given by
$K(\textbf{x},\textbf{y})=k(\textbf{x}-\textbf{y})$ with
$k(\textbf{x})$ having period $1$ in each component of
$\textbf{x}=(x_{1},\ldots, x_{s})$. As it is quite common in
applications of QMC-methods (see for example \cite{dick2007lattice},
\cite{sloan2001tractability}, \cite{kuo2003component}) it is assumed that   
$g$ and $k$ belong to a weighted Korobov space. Of course, there
exists a vast literature concerning the numerical solution of
Fredholm equations, see for instance \cite{ikebe1972galerkin}, \cite{atkinson1967numerical} or \cite{twomey1963numerical}.
In particular, we want to mention the work of I. Sloan in the late 1980's where he
explored various quadrature rules for solving integral equations
and applications to engineering problems (\cite{sloan1989representation}, \cite{sloan1988quadrature} and \cite{kumar1987new}), which have also, after some modifications, been applied to Volterra type integral equations (see \cite{brunner1984iterated} or \cite{brunner1992implicitly}).
In \cite{dick2007lattice} 
the authors approximate $f$ using the Nystr\"om method based on QMC
rules.\\
For points $\textbf{t}_{1},\ldots, \textbf{t}_{N}$ in
$[0,1]^{s}$ the $N$-th approximation of $f$ is given by
\begin{align}
f_{N}(\textbf{x}):=g(\textbf{x})+\frac{1}{N}\sum^{N}_{n=1}K(\textbf{x},\textbf{t}_{n})f_{N}(\textbf{t}_{n}),
\end{align}
where the function values $f_{N}(\textbf{t}_{1}),\ldots ,
f_{N}(\textbf{t}_{N})$ are obtained by solving the linear system
\begin{align}\label{eq:linearsystem}
f_{N}(\textbf{t}_{j})=
g(\textbf{t}_{j})+\frac{1}{N}\sum^{N}_{n=1}K(\textbf{t}_{j},
\textbf{t}_{n})f_{N}(\textbf{t}_{n}),\; j=1,\ldots ,N.
\end{align}
Under some mild conditions on $K, N,$ and the integration points
$\textbf{t}_{1},\ldots , \textbf{t}_{N},$ it is shown in \cite{dick2007lattice}
that there exists a unique solution of (\ref{eq:linearsystem}). Furthermore, the
authors analyze the worst case error of this, so-called
QMC-Nystr\"om method. In addition, good lattice point sets
$\textbf{t}_{1},\ldots , \textbf{t}_{N}$ are presented, which lead
to a best possible worst case error. A special focus of this
important paper lies on the study of tractability and strong
tractability of the QMC-Nystr\"om method. For tractability theory in
general we refer to the fundamental monograph of
\cite{novak2010tractability}.
Using ideas of E. Hlawka \cite{hlawka1961funktionen} the third
author of the present paper worked on iterative methods for solving Fredholm and Volterra equations,
see also Hua-Wang \cite{hua2012applications}.\\

The idea is to approximate the solution of integral equations by means of iterated (i.e.
multi-dimensional) integrals. The convergence of this procedure
follows from Banach's fixed point theorem and error estimates can
be established following the proof of the Picard-Lindel\"of
approximation for ordinary differential equations. To be more
precise, let us consider integration points
$\textbf{t}_{1},\ldots , \textbf{t}_{N}\in [0,1]^{s}$ with star discrepancy $D^*_{N}$
defined as usual by
\begin{align}
D^*_{N}=\sup_{J\subset[0,1]^{s}}\left|\frac{1}{N}\sharp\{n\leq
N:\textbf{t}_{n}\in J\}-\lambda(J)\right|,
\end{align}
where the supremum is taken over all axis-aligned boxes $J$ with one vertex in the origin and Lebesgue measure $\lambda(J)$.
In \cite{Tichy1984_Integral} the following system of $r$ 
integral equations has been considered for given functions $g_j$ on $[0,1]^{s+r}$ and $h_j$ on $[0,1]^s$:
\begin{align}\label{eq:system}
f_{j}(\textbf{x})= \int^{x_{1}}_{0}\ldots
\int^{x_{s}}_{0} g_{j}(\xi_{1},\ldots , \xi_{s},f_{1}(\boldsymbol{\xi}),\ldots,f_{r}(\boldsymbol{\xi}))d\xi_{s}\ldots d \xi_{1}+ h_{j}(\textbf{x})\,,\;
j=1,\ldots,r
\end{align}
where we have used the notations $\textbf{x}=(x_{1},\ldots ,x_{s})\in [0,1]^{s}$ and $\boldsymbol{\xi}=(\xi_1,\ldots,\xi_s)$.
Furthermore, we assume that the partial derivatives up to order $s$ of the functions
$g_{j}$ and $h_{j}$, $j=1,\ldots,r$, are bounded by some constants $G$ and $H$, respectively.
Then,
for a given point set $\textbf{t}_{1},\ldots , \textbf{t}_{N}$ in
$[0,1]^{s}$ with discrepancy $D^*_{N}$,
the solution $\textbf{f}=(f_{1},\ldots,f_{r})$ of the system (\ref{eq:system}) can be approximated by
the quantities
$\textbf{f}^{(k)}=(f^{(k)}_{1},\ldots,f^{(k)}_{r}),$
given recursively by
\begin{align}
f^{(k+1)}_{j}(\textbf{x})=\frac{x_{1}\cdots x_{s}}{N}
\sum^{N}_{n=1}g_{j}(x_1 t_{1,n},\ldots,x_s t_{s,n},f_{1}^{(k)}(\textbf{x}\cdot
\textbf{t}_{n}),\ldots,f^{(k)}_{r}(\textbf{x}\cdot\textbf{t}_{n}));
\end{align}
here $\textbf{x}\cdot\textbf{t}_{n}$ stands for the inner product
$x_{1}t_{1,n}+\ldots + x_{s}t_{s,n}$, where $\textbf{t}_{n}=
(t_{1,n},\ldots,t_{s,n})$.
In \cite{Tichy1984_Integral} it is shown, that based on     
the classical Koksma-Hlawka inequality the worst case error, i.e.,
$\parallel \textbf{f}^{(k)}- \textbf{f}\parallel_{\infty}$
(sum of componentwise supremum norms) can be estimated in terms of
the bounds $G$ and $H$ and the discrepancy $D^*_{N}$ of the integration points.
This method was also extended to integral equations with singularities, such as
Abel's integral equation. The main focus of the
present paper lies on applications in mathematical finance. In Albrecher \& Kainhofer \cite{AlbKain02}
the above method was used for the numerical solution of certain
Cram\'er-Lundberg models in risk theory. However, it turned out that
in these models certain discontinuities occur. This means, that
one cannot assume bounds for the involved partial derivatives and
simply apply the classical Koksma-Hlawka inequality. Moreover, the
involved functions are indicator functions of simplices thus not
of bounded variation in the sense of Hardy and Krause, see Drmota \& Tichy \cite{drmota1997sequences} and Kuipers
\& Niederreiter \cite{kuipers2012uniform}.

Albrecher \& Kainhofer \cite{AlbKain02} considered a risk model with non-linear dividend
barrier and made some assumptions to overcome the difficulties
caused by discontinuities. For such applications it could help to
use a different notion of variation for multivariate functions.
G\"otz \cite{gotz2002discrepancy} proved a version of the Koksma-Hlawka inequality for
general measures, Aistleitner \& Dick \cite{AistDick2015} considered functions
of bounded variation with respect to signed measures and
Brandolini et al. \cite{Brandolini2013a,Brandolini2013b}
replaced the integration domain $[0,1]^{s}$ by an arbitrary
bounded Borel subset of $\mathbb{R}^{s}$ and proved the inequality
for piecewise smooth integrands. Based on fundamental work of
Harman \cite{Harman2010}, a new concept of variation was developed for a wide
class of functions, see Pausinger \& Svane \cite{PausingerSvane2015} and Aistleitner et al. \cite{AistPausSvanTich2016}.\\
In the following we give a brief overview on concepts of
multivariate variation and how they can be applied for error
estimates in numerical integration. Let $f(\textbf{x})$ be a
function on $[0,1]^{s}$ and $\textbf{a}= (a_{1},\ldots ,
a_{s})\leq\textbf{b}= (b_{1},\ldots , b_{s})$ points
in $[0,1]^{s}$, where $\leq$ denotes the natural componenwise partial
order. Following the notation of Owen \cite{Owen2005} and Aistleitner et al. \cite{AistPausSvanTich2016} for a subset
$u\subseteq\{1,\ldots,s\}$ we denote by
$\textbf{a}^{u}:\textbf{b}^{-u}$ the point with $i$th coordinate
equal to $a_{i}$ if $i\in u$ and equal to $b_{i}$ otherwise. Then
for the box $R=[\textbf{a},\textbf{b}]$ we introduce the
$s-$dimensional difference operator
$$\Delta^{(d)}(f;R)=\Delta(f;R)=\sum_u(-1)^{\vert u
\vert}f(\textbf{a}^{u}:\textbf{b}^{-u}),$$
where the summation is extended over all subsets
$u\in\{1,\ldots,s\}$ with cardinality $\vert u \vert$ and complement
$-u$. Next we define partitions of $[0,1]^{s}$ as they are
used in the theory of multivariate Riemann integrals, which we call
here \emph{ladder}. A ladder $\mathcal{Y}$
in $[0,1]^{s}$ is the cartesian product of one-dimensional
partitions $0=y_{1}^{j}<\ldots <y_{k_{j}}^{j}<1$ (in any dimension
$j=1,\ldots,s$). Define the successor $(y_{i}^{j})_+$ of
$y_{i}^{j}$ to be $y_{i+1}^{j}$ if $i<k_{j}$ and
$(y_{k_{j}}^{j})_+=1$. For $\textbf{y}=
(y_{i_{1}}^{1},\ldots,y_{i_{s}}^{s})\in\mathcal{Y}$ we define the
successor $\textbf{y}_{+}=
((y_{i_{1}}^{1})_{+},\ldots,(y_{i_{s}}^{s})_{+})$ and have

$$\Delta(f;[0,1]^{s})=\sum_{{\textbf{y}}\in\mathcal{Y}}\Delta(f;[\textbf{y},\textbf{y}_{+}]).$$

Using the notation

$$V_{\mathcal{Y}}(f;[0,1]^{s})=\sum_{{\textbf{y}}\in\mathcal{Y}}\Delta(f;[\textbf{y},\textbf{y}_{+}])$$

the Vitali variation of $f$ over $[0,1]^{s}$ is defined by
\begin{align}\label{eq:Vitali}
V(f;[0,1]^{s})=\sup_{\mathcal{Y}}V_{\mathcal{Y}}(f;[0,1]^{s}).
\end{align}
Given a subset $u\subseteq \{1,\ldots,s\},$ let

$$\Delta_{u}(f;[\textbf{a},\textbf{b}])=\sum_{v\subseteq u}(-1)^{\vert v\vert}f(\textbf{a}^{v}:\textbf{b}^{-v})$$

and set $\textbf{0}=(0,\ldots,0), \textbf{1}=
(1,\ldots, 1)\in [0,1]^{s}$. For a ladder $\mathcal{Y}$ there is a
corresponding ladder $\mathcal{Y}_{u}$ on the $\vert u\vert$-dimensional face of $[0,1]^{s}$
consisting of points of the
form $\textbf{x}^{u}:\textbf{1}^{-u}$. Clearly,

$$\Delta_{u}(f; [0,1]^{s})=\sum_{{\textbf{y}}\in\mathcal{Y}_{u}}\Delta_{u}(f;[\textbf{y},\textbf{y}_{+}]).$$

Using the notation

$$V_{\mathcal{Y}_u}(f;[0,1]^{s})=\sum_{{\textbf{y}}\in\mathcal{Y}_{u}}\Delta_{u}(f;[\textbf{y},\textbf{y}_{+}])$$

for the variation over the ladder $\mathcal{Y}_{u}$ of the
restriction of $f$ to the face of $[0,1]^{s}$ specified by $u$,
the Hardy-Krause variation is defined as

$$\mathcal{V}(f)=\mathcal{V}_{HK}(f;[0,1]^{s})=\sum_{\emptyset\neq u\subseteq\{1,\ldots,s\}}\sup_{\mathcal{Y}_{u}}V_{\mathcal{Y}_u}(f;[0,1]^{s}).$$

Assuming that $f$ is of bounded Hardy-Krause variation, the
classical Koksma-Hlawka inequality reads as follows:
\begin{align}\label{eq:KH}
\left\vert\frac{1}{N}\sum_{n=1}^{N}f(\textbf{x}_{n})-\int_{[0,1]^s} f(\textbf{x})d\textbf{x}\right\vert\leq\mathcal{V}(f)D^*_{N},
\end{align}
where $\textbf{x}_{1},\ldots,\textbf{x}_{N}$ is a finite point set
in $[0,1]^{s}$ with star discrepancy $D^*_{N}$.
In the case $f:[0,1]^s\to\Real$ has continuous mixed partial derivatives up to order $s$ the Vitali variation (\ref{eq:Vitali}) is given by
\begin{align}\label{eq:varder}
\mathcal{V}(f;[0,1]^s)=\int_{[0,1]^s}\left\vert\frac{\partial^s f}{\partial x_1\cdots\partial x_s}(\textbf{x})\right\vert d\textbf{x}.
\end{align}
Summing over all non-empty subsets $u\subseteq[0,1]^s$ immediately yields an explicit formula for the Hardy-Krause variation in terms of
intergrals of partial derivatives, see Leobacher \& Pillichshammer \cite[Ch.3, p. 59]{LeobacherPillichshammer2014}. In particluar,
the Hardy-Krause variation can be estimated from above by an absolute constant if we
know global bounds on all partial derivatives up to order $s$.\\
\\
In the remaining part of the introduction we briefly sketch a more general concept of
multidimensional variation which was recently developed in \cite{PausingerSvane2015}.
Let $\mathcal{D}$ denote an arbitrary family of measurable subsets
of $[0,1]^{s}$ which contains the empty set $\emptyset$ and
$[0,1]^{s}$. Let $\mathcal{L}(\mathcal{D})$ denote the
$\Real-$vectorspace generated by the system of indicator
functions $\mathds{1}_{A}$ with $A\in\mathcal{D}$.\\
\\
A set
$A\subseteq[0,1]^{s}$ is called an algebraic sum of sets in
$\mathcal{D}$ if there exist $A_{1},\ldots , A_{m}\in \mathcal{D}$
such that
$$\mathds{1}_{A}=\sum_{i=1}^{n}\mathds{1}_{A_{i}}-\sum_{i=n+1}^{m}\mathds{1}_{A_{i}},$$

and $\mathcal{A}$ is defined to be the collection of algebraic sums
of sets in $\mathcal{D}$. As in \cite{PausingerSvane2015} we define the Harman complexity
$h(A)$ of a non-empty set $A\in\mathcal{A}, A \neq [0,1]^{s}$ as
the minimal number $m$ such there exist $A_{1},\ldots , A_{m}$
with

$$\mathds{1}_{A}=\sum_{i=1}^{n}\mathds{1}_{A_{i}}-\sum_{i=n+1}^{m}\mathds{1}_{A_{i}},$$

for some $1\leq n \leq m$ and $A_{i}\in\mathcal{D}$ or
$[0,1]^{s}\setminus A_{i}\in\mathcal{D}$. Moreover, set
$h([0,1]^{s})= h(\emptyset)= 0$ and for $f\in\mathcal{L}(\mathcal{D})$

$$V^{\ast}_{\mathcal{D}}(f)=\inf\left\{\;\sum_{i=1}^{m}\vert \alpha_{i}\vert h_{\mathcal{D}(A_{i})}:
f=\sum_{i=1}^{m}\alpha_{i}{\mathds{1}}_{A_{i}},\;
\alpha_{i}\in\mathbb{R},\; A_{i}\in\mathcal{D}\;\right\}.$$

Furthermore, let $\mathcal{V}_{\infty}(\mathcal{D})$ denote the
collection of all measurable, real-valued functions on $[0,1]^{s}$
which can be uniformly approximated by functions in
$\mathcal{L}(\mathcal{D}).$ Then the $\mathcal{D}-$variation of
$f\in\mathcal{V}_{\infty}(\mathcal{D})$ is defined by
\begin{equation}\label{variation}
 V_{\mathcal{D}}(f)=\inf\{\;\liminf_{i\rightarrow\infty}V_{\mathcal{D}}^{\ast}(f_{i}):
\;f_{i}\in\mathcal{L}(\mathcal{D}),\;f=\lim_{i\rightarrow\infty}f_{i}\;\},
\end{equation}

and set $V_{\mathcal{D}}(f)=\infty$ if $f\notin
\mathcal{V_{\infty}}\mathcal{(D)}.$ The space of functions of
bounded $\mathcal{D}-$variation is denoted by
$\mathcal{V}(\mathcal{D})$. Important classes of sets
$\mathcal{D}$ are the class $\mathcal{K}$ of convex sets and the
class $\mathcal{R}^{\ast}$ of axis aligned boxes containing
$\bf{0}$ as a vertex. In Aistleitner et al. \cite{AistPausSvanTich2016} it is shown that the
Hardy-Krause variation $\mathcal{V}(f)$ coincides with
$\mathcal{V}_{\mathcal{R}^{\ast}}(f)$. For various applications
the $\mathcal{D}-$variation seems to be a more natural and
suitable concept. A convincing example concerning an application to
computational geometry is due to Pausinger \& Edelsbrunner \cite{edelsbrunner2016approximation}.
Pauisnger \& Svane \cite{PausingerSvane2015} considered the variation $\mathcal{V}_{\mathcal{K}}(f)$ with respect to the class of convex sets.
They proved the following version of the Koksma-Hlawka inequality:
\begin{align*}
\left\vert\frac{1}{N}\sum_{n=1}^{N}f(\textbf{x}_{n})-\int_{[0,1]^s} f(\textbf{x})d\textbf{x}\right
\vert\leq \mathcal{V}_{\mathcal{K}}(f)\tilde{D}_{N},
\end{align*}
where $\tilde{D}_{N}$ is the isotropic discrepancy of the point set $\textbf{x}_{1},\ldots,\textbf{x}_{N}$, which is defined
as follows
\begin{align*}
\tilde{D}_{N}=\sup_{C\subset\mathcal{K}}\left|\frac{1}{N}\sharp\{n\leq
N:\textbf{x}_{n}\in C\}-\lambda(C)\right|.
\end{align*}
Pausinger \& Svane \cite{PausingerSvane2015} have shown that twice continuously
differentiable functions $f$ admit finite $\mathcal{V}_{\mathcal{K}}(f)$,
and in addition they gave a bound which will be usefull in our context.\\
\\
Our paper is structured as follows. In Section 2 we introduce specific Markovian models in risk theory where in a natural
way integral equations occur. These equations are based on arguments from renewal theory
and only in particular cases they can be solved analytically. In Section 3 we develop a QMC method for such equations. We give an 
error estimates based on Koksma-Hlawka type inequalities for such models.
In Section 4 we compare our numerical results to exact solutions in specific instances.


\section{Discounted penalties in the renewal risk model}
\subsection{Stochastic modeling of risks}
In the following we assume a stochastic basis $(\Omega,\,\mathcal{F},\,P)$
which is large enough to carry all the subsequently defined random variables.
In risk theory the surplus process of an insurance portfolio is modeled by a stochastic process
$X=(X_t)_{t\geq 0}$. In the classical risk model, going back to Lundberg \cite{Lundberg1903}, $X$ takes the form
\begin{align}\label{eq:CLmodel}
X_t=x+c\,t-\sum_{i=1}^{N_t}Y_i,
\end{align}
where the deterministic quantities $x\geq 0$ and $c\geq 0$ represent the initial capital and the premium rate.
The stochastic ingredient $S_t=\sum_{i=1}^{N_t}Y_i$ is the cumulated claims process which is a compound Poisson process.
The jump heights - or claim amounts - are $\{Y_i\}_{i\in\mathbb{N}}$ for which $Y_i\stackrel{iid}{\sim} F_Y$ with $F_Y(0)=0$.
The counting process $N=(N_t)_{t\geq 0}$ is a homogeneous Poisson process with intensity $\lambda>0$.
A crucial assumption in the classical model is the independence between $\{Y_i\}_{i\in\mathbb{N}}$ and $N$.
A major topic in risk theory is the study of the ruin event.
We introduce the time of ruin $\tau=\inf\{t\geq 0\,\vert\, X_t<0\}$, i.e.,
the first point in time at which the surplus becomes negative.
In this setting $\tau$ is a stopping time with respect to
the filtration generated by $X$, $\{\mathcal{F}_t^X\}_{t\geq 0}$ with $\mathcal{F}_t^X=\sigma\{X_s\,\vert\,0\leq s\leq t\}$.
A first approach for quantifying the risk of $X$, is the study of the associated ruin probability
\begin{align*}
\psi(x)=P_x(X_t<0\;\mbox{for some}\,t\geq 0)=P_x(\tau<\infty),
\end{align*}
which is non-degenerate if $\mathbb{E}_x(X_1)>0$, and satisfies the integral equation
\begin{align*}
\frac{c}{\lambda}\psi(x)=\int_x^\infty 1-F_Y(y)dy+\int_0^x\psi(x-y)(1-F_Y(y))dy.
\end{align*}
In Gerber \& Shiu \cite{GerShiuTOR98,Gerber_Shiu05} so-called discounted penalty functions are introduced.
This concept allows for an integral ruin evaluation and is based on a function
$w:\mathbb{R}^+\times\mathbb{R}^+\to\mathbb{R}$ which links the deficit at ruin $\vert X_\tau \vert$ and the surplus prior to ruin $X_{\tau-}:=\lim_{t\nearrow\tau}X_t$
via the function
\begin{align*}
V(x)=\mathbb{E}_x\left(e^{-\delta \tau} w(\vert X_\tau\vert,X_{\tau-})\,\mathds{1}_{\{\tau<\infty\}}\right).
\end{align*}
The time of ruin $\tau$ is included by means of a discounting factor $\delta>0$ which gives more weight to an early ruin event.
In this setting specific choices of $w$ allow for an unified treatment of ruin related quantities.
\begin{remark}
When putting a focus on the study of $\psi(x)$, the condition $\mathbb{E}_x(X_1)>0$ is crucial.
It says that on average premiums exceed claim payments in one unit of time.
Standard results, see Asmussen \& Albrecher \cite{AsAlb2010}, show that under this condition $\lim_{t\to\infty}X_t=+\infty$ $P$-a.s.
From an economic perspective the accumulation of
an infinte surplus is unrealistic and risk models including shareholder participation via dividends are introduced in the literature.
We refer to \cite{AsAlb2010} for model extensions in this direction.
\end{remark}

\subsection{Markovian risk model}\label{markov_risk}
In the following we consider an insurance surplus process $X=(X_t)_{t\geq 0}$ of the form
\begin{align*}
X_t=x+\int_0^t c(X_{s-})ds-\sum_{i=1}^{N_t}Y_i. 
\end{align*}
The quantity $x\geq 0$ is called the initial capital, the cumulated claims are represented by $S_t=\sum_{i=1}^{N_t}Y_i$
and the state-dependend premium rate is $c(\cdot)$. The cumulated claims process $S=(S_t)_{t\geq 0}$ is given by a
sequence $\{Y_i\}_{i\in\Natural}$ of positive, independently and identically distributed (iid) random variables and a counting process $N=(N_t)_{t\geq 0}$.
For convenience we assume that the claims distribution admits a continuous density $f_Y:\Real^+\to\Real^+$.
In our setup we model the claim counting process $N=(N_t)_{t\geq 0}$ as a renewal counting process which is specified by the
inter-jump times $\{W_i\}_{i\in\mathbb{N}}$ which are positive and iid random variables. Then, the time of the $i-$th jump
is $T_i=W_1+\ldots+W_i$ and if we assume that $W_1$ admits a density $f_W$,
the jump intensity of the process $X$ is $\lambda(t')=\frac{f_W(t')}{1-\int_0^{t'} f_W(s)ds}$. Here $t'$ denotes the \emph{time since the last jump}.
A common assumption we are going to adopt, is the independence between $\{Y_i\}_{i\in\Natural}$ and $\{W_i\}_{i\in\mathbb{N}}$.\\
We choose, in contrast to classical models, a non-constant premium rate to model the effect of a so-called dividend barrier $a>0$ in a smooth way.
A barrier at level $a>0$ has the purpose that every excess of surplus of this level is distributed as a dividend to shareholders which allows to include
economic considerations in insurance modeling. Mathematically, this means that the process $X$ is reflected at level $a$.
Now instead of directly reflecting the process we use the following construction. Fix $\varepsilon>0$ and for some $\tilde{c}>0$, define
\begin{align}
c(x)=\left\{
\begin{array}{ll}
\tilde{c}, & x\in[0,a-\varepsilon),\\
f(x), & x\in[a-\varepsilon,a],\\
0,& x>a,
\end{array}\right. 
\end{align}
with some positive and twice continuously differentiable function $f$ which fulfills $f(a-\varepsilon)=\tilde{c},\,f(a)=0,\,f'(a-\varepsilon)=f'(a)=f''(a-\varepsilon)=f''(a)=0$.
Altogether, we assume $c(\cdot)\in\mathcal{C}^2[0,a]$ with some Lipschitz constant $L>0$ and
$c'(a-\varepsilon)=c'(a)=0$, $c''(a-\varepsilon)=c''(a)=0$, $c'\leq 0$ and bounded derivatives $c'$, $c''$.
Then $\lim_{x\nearrow a}c(x)=0$ and the process always stays below level $a$ if started in $[0,a)$.\\
A concrete choice for $f$ would be
\begin{align}\label{eq:polynomialchoice}
\frac{c (a-x)^3 \left(15 \varepsilon  (x-a)+6 (a-x)^2+10 \varepsilon^2\right)}{\varepsilon^5}.
\end{align}
In the following we do not specify $f$ any further.
\\
In this setting we add $X_0=x$ into the definition of the time of ruin, i.e., $\tau_x=\inf\{t\geq 0\,\vert\, X_t<0,\,X_0=x\}$.
\begin{remark}
In this model setting ruin can only take place at some jump time $T_k$ and since the process is bounded a.s. we have that $P_x(\tau_x<\infty)=1$.
If an approximation to classical reflection of the process at level $a$ is implemented, then the process virtually started above $a$ is forced to jump
down to $a-\varepsilon$ and continue from this starting value. Consequently, we put the focus on starting values $x\in[0,a)$.
\end{remark}
In the remainder of this section we will study analytic properties of the discounted value function which in this framework takes the form
\begin{align}\label{eq:discpen}
V(x)=\mathbb{E}_x\left(e^{-\delta \tau_x} w(\vert X_\tau\vert,X_{\tau-})\right),
\end{align}
with $\delta>0$ and a continuous penalty function $w:\mathbb{R}^+\times[0,a)\to\mathbb{R}$.\\
To have a well defined function, typically the following integrability condition is used
\begin{align*}
\int_0^\infty\int_0^\infty w(x,y)f_Y(x+y)dy\,dx<\infty,
\end{align*}
see \cite{AsAlb2010}. Since our process is kept below level $a$ and $w$ is supposed to be continuous in both arguments we can naturally replace
the above condition by
\begin{align}\label{eq:intcond}
\sup_{z\in[0,a)}\int_0^\infty \vert w(|z-y|,z)\vert f_Y(y)dy=:M<\infty,
\end{align}
which we will assume in the following. The condition from equation (\ref{eq:intcond}) holds true for example,
if $\vert w(x,y)\vert\leq (1+\vert x\vert+\vert y\vert)^p$ and $F_Y$ admits a finite $p$-th moment for some $p\geq 1$.\\
\begin{remark}
From the construction of $X$ we have that $\tilde{X}=(\tilde{X}_t)_{t\geq 0}$ with $\tilde{X}_t=(X_t,t'(t),t)$ is a piecewise-deterministic Markov process,
see Davis \cite{DavisMarkovOpt93}. Since the jump intensity depends on $t'=t-T_{N_t}$, one needs this additional component for the \emph{Markovization} of $X$.
But on the discrete time skeleton $\{T_i\}_{i\in\Natural}$ with $T_0=0$ the process $X=\{X_{T_k}\}_{k\in\Natural}$ has the Markov property.
\end{remark}

\subsection{Analytic properties and a fixed point problem}
We start with showing some elementary analytical properties of the function $V$ defined in (\ref{eq:discpen}).
\begin{theorem}
The function $V:[0,a)\to\Real$ is bounded and continuous.
\end{theorem}
\begin{proof}
The boundedness of $V$ follows directly from the assumption made in (\ref{eq:intcond}).\\
For proving continuity we split off the expectation defining $V$ into two parts which we separately deal with.
Let $x>y$ and observe
\begin{align*}
\left\vert\,V(x)-V(y)\,\right\vert=&
\left\vert\,\mathbb{E}\left[e^{-\delta\tau_x}w(\vert X_{\tau_x}^x\vert,X_{\tau_x-}^x)-e^{-\delta \tau_y}w(\vert X_{\tau_y}^y\vert,X_{\tau_y-}^y)\right]\right\vert
\\
\leq& \mathbb{E}\left[e^{-\delta\tau_x}\left\vert w(\vert X_{\tau_x}^x\vert,X_{\tau_x-}^x)-w(\vert X_{\tau_x}^y\vert,X_{\tau_x-}^y)\right\vert \mathds{1}_{\{\tau_x=\tau_y\}}\right]\\
&+\mathbb{E}\left[\left\vert e^{-\delta\tau_x} w(\vert X_{\tau_x}^x\vert,X_{\tau_x-}^x)-e^{-\delta\tau_y}w(\vert X_{\tau_y}^y\vert,X_{\tau_y-}^y)\right\vert \mathds{1}_{\{\tau_x>\tau_y\}}\right]\\
=&A+B.
\end{align*}
For $A$ we fix some $T>0$ and notice the following bound
\begin{align}\label{eq:estA}
A\leq \mathbb{E}\left[e^{-\delta\tau_x}\vert w(\vert X_{\tau_x}^x\vert,X_{\tau_x-}^x)-w(\vert X_{\tau_x}^y\vert,X_{\tau_x-}^y)\vert \mathds{1}_{\{\tau_x=\tau_y\leq T\}}\right]
+2M\,P(\tau_x>T)\leq 2M.
\end{align}
Before going on we need some estimates on the difference of two paths, one starting in $x$ and the other in $y$.
For fixed $\omega\in\Omega$ we have that on $(0,T_1(\omega)$ the surplus fulfills $\frac{\partial X_t(\omega)}{\partial t}=c(X_t(\omega))$
with initial condition $X_0=0$, $T_1(\omega)$ is finite with probability one. Standard arguments on ordinary differential equations, see for instance
Stoer \& Bulirsch \cite[Th. 7.1.1 - 7.1.8]{StoerBulirsch2000}, yield that an appropriate solution exists and is continuously differentiable in $t$
and continuous in the initial value $x$. We even get the bound $\vert X_t^x-X_t^y\vert \leq e^{L\,t}\,\vert x-y\vert$ for fixed $\omega$,
where $X_t^x$ denotes the path which starts in $x$ and $L>0$ the Lipschitz constant of $c(\cdot)$. 
From these results we directly obtain for a given path
\begin{align*}
\vert X_{T_1-}^x-X_{T_1-}^y\vert\leq e^{L {T_1}}\vert x-y\vert,
\end{align*}
which by iteration results in
\begin{align*}
\vert X_{T_n-}^x-X_{T_n-}^y\vert=\vert X_{T_n}^x-X_{T_n}^y\vert\leq e^{L {T_n}}\vert x-y\vert,
\end{align*}
because $\vert X_{T_n}^x-X_{T_n}^y\vert=\vert X_{T_n-}^x-Y_n-(X_{T_n-}^y-Y_n)\vert=\vert X_{T_n-}^x-X_{T_n-}^y\vert$.
Since ruin takes place at some claim occurrence time $T_k$ we get that on
$\{\omega\in\Omega\,\vert\,\tau_x=\tau_y\leq T\}$ the quantities $\vert X_{\tau_x}^x\vert$ and $X_{\tau_x-}^x$
converge to the corresponding quantities started in $y$, all possible differences are bounded by $e^{LT}\vert x-y\vert$.
Therefore, sending $y$ to $x$ in (\ref{eq:estA}) and then sending $T$ to infinity, we get that $A$ converges to zero
because $P(\tau_x<\infty)=1$ and bounded convergence. We can repeat the argument for $x\to y$ when using $P(\tau_y>T)$ in (\ref{eq:estA}).\\
\\
Now consider part $B$. We first observe that $B\leq 2M P(\tau_x>\tau_y)$.
Consequently, we need to show that $P(\tau_x>\tau_y)$ tends to zero if $y\to x$ or $x\to y$.
Again, fix $\omega\in\Omega$ for which $\tau_x(\omega)>\tau_y(\omega)$, this implies
that there is a claim amount $Y_n$, occuring at some point in time $T_n$, for which
\begin{align*}
X_{T_n-}^x(\omega)\geq Y_n(\omega)>X_{T_n-}^y(\omega),
\end{align*}
i.e., causing ruin for the path started in $y$, $(X_t^y)$, but not causing ruin for the one started in $x$, $(X_t^x)$.
From the construction of the drift $c(\cdot)$, it is decreasing to zero, we have that, surpressing the $\omega$ dependence,
\begin{align*}
0< Y_n-X_{T_n-}^y\leq X_{T_n-}^x-X_{T_n-}^y\leq x-y.
\end{align*}
Since $X_{T_n-}^y\in[0,a)$ we have
\begin{align*}
P(\tau_x>\tau_y)\leq \sup_{q\in[0,a)} P(0<Y-q\leq x-y)= \sup_{q\in[0,a)} \{F_Y(x-y+q)-F_Y(q)\},
\end{align*}
which approaches zero whenever $x$ and $y$ tend to each other since $F_Y$ is continuous.
\end{proof}
Define for functions $f\in\mathcal{C}_b([0,a))$ the operator $\mathcal{A}$ by
\begin{align}\label{eq:DefOperator}
\mathcal{A}f(x):=\mathbb{E}_x\left(e^{-\delta T_1}f(X_{T_1})\mathds{1}_{\{T_1<\tau_x\}} + e^{-\delta \tau_x}w(\vert X_{T_1}\vert,X_{T_1-})\mathds{1}_{\{\tau_x=T_1\}} \right).
\end{align}
The Markov property of the sequence $\{X_{T_i}\}_{i\in\Natural}$ and the definition of $V$ in (\ref{eq:discpen})
allow us to derive that $V=\mathcal{A}V$, or explicitely written
\begin{align*}
V(x)=\mathbb{E}_x\left[e^{-\delta T_1}V(X_{T_1})\mathds{1}_{\{T_1<\tau_x\}}+e^{-\delta T_1} w(\vert X_{T_1}\vert,X_{T_1-})\mathds{1}_{\{\tau_x=T_1\}}\right].
\end{align*}
We can state the following lemma.
\begin{lemma}
If $\delta>0$, the operator $\mathcal{A}:\mathcal{C}_b([0,a))\to \mathcal{C}_b([0,a))$ defined in (\ref{eq:DefOperator}) is a contraction with respect to
$\vert\vert\cdot\vert\vert_\infty$.
\end{lemma}
\begin{proof}
Let $f\in\mathcal{C}_b([0,a))$ be bounded by some constant $M'$, then 
\begin{align*}
\mathcal{A}f(x)=\mathbb{E}_x\left(e^{-\delta T_1}f(X_{T_1})\mathds{1}_{\{T_1<\tau_x\}}+ e^{-\delta \tau_x}w(\vert X_{T_1}\vert,X_{T_1-})\mathds{1}_{\{\tau_x=T_1\}} \right),
\end{align*}
is bounded by $\max\{M,M'\}$. From the integral representation of $\mathcal{A}f(x)$ we get continuity in $x$,
\begin{multline*}
\mathcal{A}f(x)=\\
\int_0^\infty e^{-\delta t_1}f_W(t_1)\left[\int_0^{X_{t_1-}}f(X_{t_1-}-y_1)dF_Y(y_1)+\int_{X_ {t_1-}}^\infty
w(\vert X_{t_1-}-y_1\vert,X_{t_1-})dF_Y(y_1)\right]dt_1,
\end{multline*}
where $X_{t_1-}$ is the ODE's solution up to time $t_1$ with $X_0=x$. From Stoer \& Bulirsch \cite[Th. 7.1.4]{StoerBulirsch2000} we have that $X_{t_1-}$ is continuous
in its initial value which shows that $\mathcal{A}f(x)$ is continuous in $x$.\\
Let $f,\,g\in \mathcal{C}_b([0,a))$, then
we have for all $x\in[0,a)$ that
\begin{align*}
\vert\mathcal{A}(f-g)(x)&\vert\leq \int_0^\infty e^{-\delta t_1}f_W(t_1)\int_0^{X_{t_1}}\vert f(X_{t_1}-y_1)-g(X_{t_1}-y_1)\vert dF_Y(y_1)dt_1\\
&\leq \vert\vert f-g\vert\vert_\infty\int_0^\infty e^{-\delta t_1}f_W(t_1)dt_1= \vert\vert f-g\vert\vert_\infty\mathbb{E}[e^{-\delta T_1}].
\end{align*}
Since $\delta>0$ and $T_1>0$ $P-$a.s., $\mathcal{A}$ is contractive with Lipschitz constant $\tilde{L}=\mathbb{E}[e^{-\delta T_1}]<1$.
\end{proof}
For a possible application of quasi-Monte Carlo techniques we need to examine
the structure of $\mathcal{A}$,
\begin{align*}
\mathcal{A}v(x)=&\int_0^\infty e^{-\delta t_1}f_W(t_1)\int_0^{X_{t_1-}}v(X_{t_1-}-y_1)dF_Y(y_1)dt_1+\\
&\int_0^\infty e^{-\delta t_1}f_W(t_1)\int_{X_{t_1-}}^\infty w(y_1-X_{t_1-},X_{t_1-})dF_Y(y_1)dt_1\\
=:&\;\mathcal{G}v(x)+\mathcal{H}(x).
\end{align*}
For $n\in\Natural$ the probabilistic interpretation of iterated applications of $\mathcal{A}$ is
$\mathcal{A}^n v(x)=\mathbb{E}_x\left(e^{-\delta T_n}v(X_{T_n})\mathds{1}_{\{T_n<\tau_x\}}+e^{-\delta \tau_x}w(\vert X_{\tau_x}\vert,X_{\tau_x-})\mathds{1}_{\{\tau_x\leq T_n\}}\right)$.
Using $\mathcal{G}$ and $\mathcal{H}$ we can write
\begin{align*}
\mathcal{A}^n v(x)= \mathcal{G}^n v(x)+\sum_{k=0}^{n-1}\mathcal{G}^k \mathcal{H}(x),
\end{align*}
where $\mathcal{G}^n v(x)=\mathbb{E}_x(e^{-\delta T_n}v(X_{T_n})\mathds{1}_{\{T_n<\tau\}})$ and
\begin{align*}
\mathcal{G}^{k-1}\mathcal{H}(x)=&\int_0^\infty\cdots\int_0^\infty\int_{X_{\bar{t}_k-}}^\infty\int_0^{X_{\bar{t}_{k-1}-}}\cdots\int_0^{X_{\bar{t}_1-}}\\
&\left(\prod_{i=1}^k e^{-\delta t_i}\,f_W(t_i)\right) w(y_k-X_{\bar{t}_k-},X_{\bar{t}_k-})dF_Y(y_k)\cdots dF_Y(y_1)dt_k\cdots dt_1.
\end{align*}
Here, $\bar{t}:=\sum_{i=1}^k t_i$ and represents the time of the $k$-th jump. We see that via $X_{\bar{t}_k}=X_{\bar{t}_{k-1}}-y_{k-1}+\int_{\bar{t}_{k-1}}^{\bar{t}_k}c(X_s)ds$ the path of the process depends on all integration variables
$(t_1,\ldots,t_k,y_1,\ldots,y_k)$.\\
For dealing with the situation $\delta=0$, i.e., when the contraction argument fails, we can use a probabilistic argument.
Since $\lim_{n\to\infty}T_n=\infty$ and $P(\tau_x<\infty)=1$ we have that
$\lim_{n\to\infty}\mathcal{G}^n v(x)=\mathbb{E}_x\left(e^{-\delta T_n}v(X_{T_n})\mathds{1}_{\{T_n<\tau\}}\right)=0$ for $v\in\mathcal{C}_b([0,a))$.
Using $\vert \mathcal{A}^nv(x)-V(x)\vert=\vert\,\mathcal{G}^nv(x)-\mathcal{G}^nV(x)\,\vert$ we get
$\lim_{n\to\infty}\mathcal{A}^n v(x)=V(x)$ pointwise, even in the case if $\delta=0$.\\
In what follows we put the focus on the determination of $\mathcal{G}^k\mathcal{H}(x)$.

\section{Approximation procedure}\label{approximation_section}
For the application of QMC methods we need to transform in a first step the integration domain in
\begin{align*}
\mathcal{G}^{k-1}\mathcal{H}(x)=&\int_0^\infty\cdots\int_0^\infty\int_{X_{\bar{t}_k-}}^\infty\int_0^{X_{\bar{t}_{k-1}-}}\cdots\int_0^{X_{\bar{t}_1-}}\\
&\left(\prod_{i=1}^k e^{-\delta t_i}\,f_W(t_i)\right) w(y_k-X_{\bar{t}_k-},X_{\bar{t}_k-})dF_Y(y_k)\cdots dF_Y(y_1)dt_k\cdots dt_1
\end{align*}
to $[0,1]^{2k}$. This is achieved by use of the following substitutions
\begin{align*}
 \alpha_i&:=e^{-t_i}\Rightarrow t_i=-\log\alpha_i\qquad\text{for }i\in\{1,\dots,k\}\\
 \beta_i&:=\frac{y_i}{X_{\bar{t}_i-}}\Rightarrow y_i= X_{\bar{t}_i-}\beta_i\qquad\text{for }i\in\{1,\dots,k-1\}\\
 \beta_k&:=e^{X_{\bar{t}_k-}}e^{-y_k}\Rightarrow y_k=X_{\bar{t}_k-}-\log\beta_k.
\end{align*}
Here it has to be taken into account that the values of the reserve process $X$ have to be calculated recursively, i.e., $X_{\bar{t}_i-}$ depends
on $t_1,\ldots,t_i$ and $y_1,\ldots,y_{i-1}$. Since the Jacobian matrix of this transformation has a lower triangular form, the determinant
can easily be found as $\frac{1}{\alpha_1\dots\alpha_k}X_{\bar{t}_1-}\cdots X_{\bar{t}_{k-1}-}\frac{1}{\beta_{k}}$. Alltogether, we arrive at
\begin{align*}
\mathcal{G}^{k-1}\mathcal{H}(x)=&\int_{[0,1]^{2k}}\prod_{i=1}^k\alpha_i^\delta\,f_W(t_i(\alpha_i))
\prod_{i=1}^kf_Y(y_i(\alpha_1,\dots,\alpha_i,\beta_1,\dots,\beta_i))\\
&\frac{1}{\alpha_1\dots\alpha_k}\,X_{\bar{t}_1-}\cdots X_{\bar{t}_{k-1}-}
\,\frac{1}{\beta_k}\,d\alpha_1\dots d\alpha_kd\beta_1\dots d\beta_k.
\end{align*}
Consequently, for recovering the Koksma-Hlawka type errorbound we need to examine the variation of the integrand:
\begin{align}\label{eq:Integrand}
F(\alpha_1,\ldots,\alpha_k,\beta_1,\ldots,\beta_k)=&
\left(\prod_{i=1}^{k-1}\alpha_i^{\delta-1}f_W(-\log(\alpha_i)\right)\left(\prod_{i=1}^{k-1}f_Y(\beta_i X_{\bar{t}_i-})X_{\bar{t}_i-}\right)\nonumber\\
&\left(\alpha_k^{\delta-1}f_W(-\log(\alpha_k))f_Y(X_{\bar{t}_k-}-\log(\beta_k))\frac{1}{\beta_k}w(-\log\beta_k,X_{\bar{t}_k-}\right).
\end{align}
Here we denote by $\phi(t,s)$ the solution to $\frac{\partial}{\partial t}x(t)=c(x(t))$ with $x(0)=s$. Consequently, we can write 
\begin{align*}
X_{\bar{t}_i-}=X_{\bar{t}_{i-1}-}-y_{i-1}+\phi(t_i,X_{\bar{t}_{i-1}-}-y_{i-1}).
\end{align*}
Or in terms of $\alpha_i$, putting $\hat{x}_{i-1}=X_{\bar{t}_{i-1}-}-y_{i-1}=X_{\bar{t}_{i-1}-}(1-\beta_{i-1})$ and
\begin{align}\label{eq:Path}
X_{\bar{t}_i-}=\hat{x}_{i-1}+\phi(-\log(\alpha_i),\hat{x}_{i-1}).
\end{align}
In the following proposition we show that with a particular choice of model parameters it is possible to apply results from \cite{PausingerSvane2015}
to show that the integrand in (\ref{eq:Integrand}) is in some sense of finite variation.
Its proof shows that probabilistic and deterministic model ingredients are considerably interconnected.
\begin{theorem}\label{finite_variation}
Let $f_W(t)=\lambda e^{-\lambda t}\mathds{1}_{\{t\geq 0\}}\;(\lambda>0)$, $f_Y(y)=\mu e^{-\mu y}\mathds{1}_{\{y\geq 0\}}\;(\mu>0)$, $w\equiv 1$
and $c(\cdot)$ be specified by (\ref{eq:polynomialchoice}).
Then, under the assumption $\lambda+\delta\geq 3$ and $\mu\geq 3$ the variation $\mathcal{V}_\mathcal{K}(F)$ (see \eqref{variation} with $\mathcal{D}=\mathcal{K}$) of $F$, defined in 
(\ref{eq:Integrand}), is finite.
\end{theorem}
\begin{proof}
The main idea of the proof is the application of \cite[Th. 3.12]{PausingerSvane2015}.
For this purpose we need to show that $M(F)=\sup\{\|\mbox{Hess}(F,x)\|\,\vert\,x\in[0,1]^{2 k}\}$, $\sup F$ and $\inf F$ are finite, with the implication
\begin{align*}
\mathcal{V}_\mathcal{K}(F)\leq \sup F-\inf F+M(F).
\end{align*}
Since in this theorem the operator (matrix) norm $\|\mbox{Hess}(F,x)\|$ is arbitrary we use the 2-norm and exploit the relation
\begin{align*}
\|\mbox{Hess}(F,x)\|_2\leq\left(\sum_{i=1}^{2k}\sum_{j=1}^{2k}[\mbox{Hess}(F,x)]_{ij}^2 \right)^{\frac 1 2}.
\end{align*}
We will show that $[\mbox{Hess}(F,x)]_{ij}$ is finite for all $x\in[0,1]^{2k}$. At first we observe that when taking derivatives
with respect to $\alpha_i$ and $\beta_j$, the structure of (\ref{eq:Path}) implies the appearance of the following terms:
\begin{align*}
&\frac{\partial}{\partial t}\phi(t,s)=c(\phi(t,s)),\; \frac{\partial^2}{\partial t^2}\phi(t,s)=c'(\phi(t,s))c(\phi(t,s)),\\
&\frac{\partial}{\partial s}\phi(t,s)=:y(t,s)=e^{\int_0^t c'(\phi(u,s))du},\;\frac{\partial^2}{\partial t\partial s}\phi(t,s)=c'(\phi(t,s))y(t,s),\\
&\frac{\partial^2}{\partial s^2}\phi(t,s)=:z(t,s)=y(t,s)\int_0^t c''(\phi(u,s))y(u,s)du.
\end{align*}
The functions $y,\,z$ correspond to the first and second derivative of the ODE's solution with respect to the initial value. They can be derived from the 
associated first and second order variational equations (see \cite{LectureODE}).
From our assumptions on $c(\cdot)$ we have that
$y$ is bounded by one ($c'\leq 0$) and all other derivatives including $z$ are bounded as well. The boundedness of $z$ can be derived from the
boundedness of $c''(\cdot)$ and an analysis of the growth behaviour of $y$.\\
For the structure of $[\mbox{Hess}(F,x)]_{ij}$ we can derive the following
\begin{align*}
\prod_{i=1}^k \alpha_i^{\delta+\lambda-a}\,\beta_k^{\mu-b}\,e^{-\mu(y_1+\cdots+y_{k-1}+X_{\bar{t}_k-})}\,
Q\left(\beta_1,\ldots,\beta_{k-1},\phi,\frac{\partial}{\partial t}\phi,\frac{\partial^2}{\partial t^2}\phi,
\frac{\partial}{\partial s}\phi,\frac{\partial^2}{\partial t\partial s}\phi,\frac{\partial^2}{\partial s^2}\phi\right),
\end{align*}
where $a,\,b\in\{1,2,3\}$ and a function $Q$. $Q$ is evaluated at the integration points and $\phi$ and its derivatives which themselves are evaluated
in points of the form $(-\log(\alpha_i),\hat{x}_{i-1})\in(0,\infty)\times[0,a)$ for $i\in\{1,\ldots,k\}$.
If $\phi$ and its derivatives are considered to be variables, neglecting their dependence on $\alpha_i$s and $\beta_i$s, then $Q$ is a polynomial of degree $k$.
The degree of the polynomial is produced by the recursive structure of the paths and its dependence on all previous jump times and sizes.
From this inspection we get that under the conditions $\lambda+\delta\geq 3$ and $\mu\geq 3$ all entries of the Hessian matrix are bounded. Furthermore,
the conditions on the parameters $\lambda,\,\delta,\,\mu$ ensure that $\sup F$ is finite and $\inf F=0$.
\end{proof}
%
\begin{remark}
We can combine the above result with the convergence rate from Banach's fixed point theorem and obtain for our specific situation
\begin{align*}
\left\| \sum_{k=0}^n\hat{\mathcal{G}}^k\mathcal{H}-V\right\|_\infty\leq &\left\| \sum_{k=0}^n(\hat{\mathcal{G}^k}\mathcal{H}-\mathcal{G}^k\mathcal{H})\right\|_\infty+\left\|\mathcal{A}^n-V\right\|_\infty
+\left\|\mathcal{G}^n v\right\|_\infty\\
&\leq \sum_{k=0}^n \mathcal{V}_\mathcal{K}(F^k)\tilde{D}_{N_k}+\frac{\tilde{L}^n}{1-\tilde{L}}\|\mathcal{A}v-v\|_{\infty}+M'\left(\frac{\lambda}{\delta+\lambda}\right)^n.
\end{align*}
Here $F^k$ denotes the integrand from (\ref{eq:Integrand}) in dimension $2k$, $\tilde{D}_{N_k}$ the
isotropic discrepancy of a pointset with $N_k$ elements in $[0,1]^{2k}$ and $\hat{\mathcal{G}}^k\mathcal{H}$ is the QMC approximation for $\mathcal{G}^k\mathcal{H}$. For the last term we used that $v$ is bounded by some $M'>0$ and the fact the $T_n$ follows a Gamma distribution $\Gamma(n,\lambda)$.\\
From the type of arguments we used for the proof of Theorem \ref{finite_variation}, we expect that the result holds true for $\Gamma$-distributed inter-claim times and jump heights and $w(y,z)=y^kz^l$ with similar conditions on the parameters. Hence the method is also applicable for this more general situation. A detailed study of this claim is part of future research.
\end{remark}

\section{Numerical results}
  In this section, we evaluate the integrals from Section \ref{approximation_section} by applying Monte Carlo and quasi-Monte Carlo methods for different choices of the penalty function $w$.
  \subsection{The discounted time of ruin}
  Letting $w(y,z):=1$, we arrive at $V(x)=\mathbb{E}(e^{-\delta\tau} w(|X_\tau|,X_{\tau-}))=\mathbb{E}(e^{-\delta\tau})$ which is the discounted time of ruin. Lin et al. \cite{lin2003classical} found an analytic expression for this discounted time of ruin if both the inter-arrival times of the claims and the claim sizes are exponentially distributed. To have a reference value, we also adopt these assumptions and denote the parameters of the exponential distributions with $\lambda$ for the parameter of the inter-arrival times and $\mu$ for the parameter of the claim sizes. The premium rate $c(\cdot)$ was chosen as in Section \ref{markov_risk} with $f$ from equation \eqref{eq:polynomialchoice}, with $\tilde{c}=2$, $a=3$ and $\varepsilon$ was set to $0.001$. Note that the results of Lin et al. \cite{lin2003classical} were proved for a reflected process in the classical sense, which means $c(x)=\tilde{c}$ for $x\leq a$ and $c(x)=0$ for $x>a$. Since Theorem \ref{finite_variation} requires a premium rate satisfying certain smoothness conditions, we cannot use a discontinuous $c$ and thus have a methodic error in our simulations. However, we will see that this error is, at least for small $\varepsilon$, very small.\\
  $ $\\
  We list the parameters together with the approximation values for increasing numbers of (Q)MC points and $k=20$ iterations of the algorithm in Table \ref{MC_w1_20_table}, whereas Table \ref{MC_w1_100_table} shows the approximation values for $k=100$ iterations of the algorithm. Figure \ref{MC_w1_20} and Figure \ref{MC_w1_100} show the MC points (green) with $95\%$ confidence intervals, together with QMC points from Sobol sequences (blue) and Halton sequences (orange).
  \begin{figure}[htbp]
	\begin{minipage}{0.5\textwidth} 
	\includegraphics[width=\textwidth]{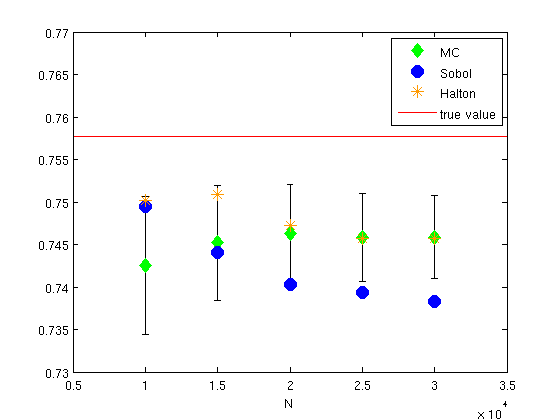}
	\caption{\footnotesize$k=20$ iterations of the algorithm.\normalsize\label{MC_w1_20}}
	\end{minipage}
	\hfill
	\begin{minipage}{0.5\textwidth}
	\includegraphics[width=\textwidth]{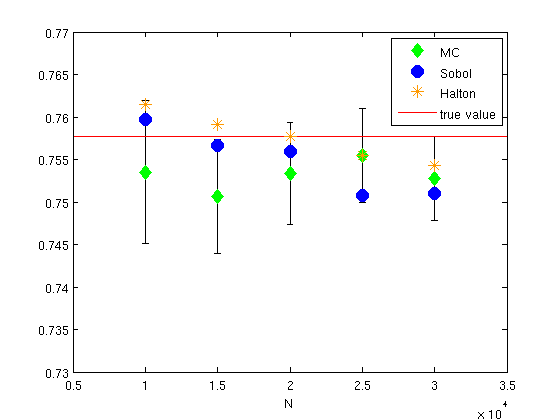}
	\caption{\footnotesize$k=100$ iterations of the algorithm.\normalsize\label{MC_w1_100}} 
	\end{minipage}
	\vskip 1cm
	\begin{minipage}{0.45\textwidth}
	\begin{adjustbox}{width=\linewidth}
	 \renewcommand{\arraystretch}{1.2}
	\begin{tabular}{rrrrrr}
	  $x$ & $\lambda$ & $\mu$ & $w(y,z)$ & $k$ & $\delta$ \\
	  $1.2$ & $1$ & $0.8$ & $1$ & $20$ & $0.05$ \\
	  \hline $N$: & $10000$ & $15000$ & $20000$ & $25000$ & $30000$ \\
	  MC: & $0.7425$ & $0.7452$ & $0.7463$ & $0.7458$ & $0.7459$\\
	  Sobol: & $0.7494 $ & $0.7440 $ & $0.7403 $ & $0.7394 $ & $0.7383$ \\
	  Halton: & $0.7502$ & $ 0.7509$ & $0.7473$ & $0.7488$ & $0.7457$
	\end{tabular}
	\end{adjustbox}\captionof{table}{\label{MC_w1_20_table}}
	\end{minipage}
	\hfill
	\begin{minipage}{0.45\textwidth}
	\begin{adjustbox}{width=\linewidth}
	\renewcommand{\arraystretch}{1.2}
	\begin{tabular}{rrrrrr}
	  $x$ & $\lambda$ & $\mu$ & $w(y,z)$ & $k$ & $\delta$ \\
	  $1.2$ & $1$ & $0.8$ & $1$ & $100$ & $0.05$ \\
	  \hline $N$: & $10000$ & $15000$ & $20000$ & $25000$ & $30000$ \\
	  MC: & $0.7535$ & $0.7507$ & $0.7534$ & $0.7555$ & $0.7527$\\
	  Sobol: & $0.7597$ & $0.7566$ & $0.7560$ & $0.7508$ & $0.7510$\\
	  Halton: & $0.7615$ & $0.7591$ & $0.7577$ & $0.7555$ & $0.7543$
	\end{tabular}
	\end{adjustbox}\captionof{table}{\label{MC_w1_100_table}}
	\end{minipage}
\end{figure}
$ $\\
The red line at height $0.7577$ marks the true value. As can be seen in Figure \ref{MC_w1_20}, the algorithm has not yet converged for $k=20$, whereas Figure \ref{MC_w1_100} shows that $k=100$ already yields a very good approximation.\\
$ $\\
To illustrate the speed of convergence, we also plotted the absolute error, both for the MC approach as well as for QMC points (again taken from Sobol and Halton sequences) for varying numbers of points $N$. Figures \ref{Diff_40} and \ref{Diff_40_bad} show the values obtained for $k=40$ iterations of the algorithm. Obviously, $k=40$ is also not yet enough to reach the actual value. But notice that the absolute error even for more iterations cannot converge to zero because of the smoothed reflection procedure.
\begin{figure}[h!]
	\begin{minipage}{0.5\textwidth} 
	\includegraphics[width=\textwidth]{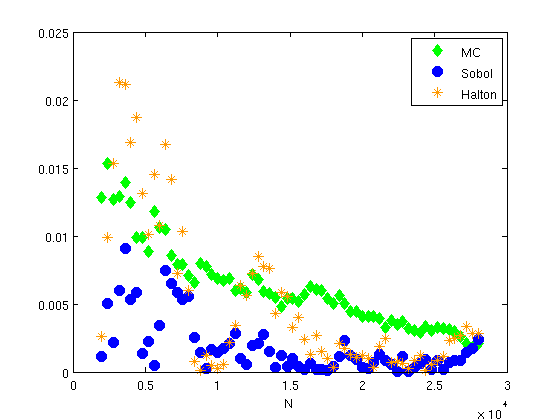}
	\caption{\footnotesize``lucky'' choice of QMC points.}
	\label{Diff_40}
	\end{minipage}
	\hfill
	\begin{minipage}{0.5\textwidth}
	\includegraphics[width=\textwidth]{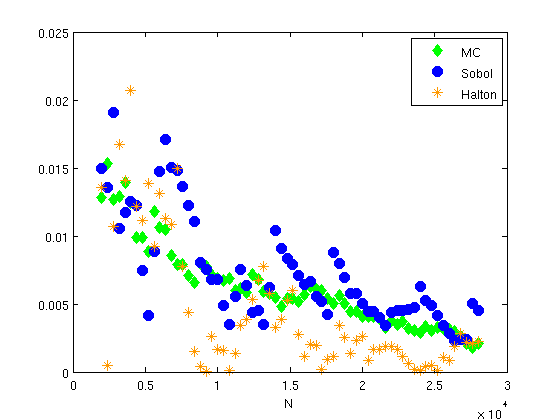}
	\caption{\footnotesize``unlucky'' choice of QMC points.\label{Diff_40_bad}} 
	\end{minipage}
\end{figure}
 For both of the QMC methods, a scramble improved the results. In the Sobol case however, an ``unlucky'' choice in the scramble and the skip value (i.e. how many elements are dropped in the beginning) can lead to relatively high variation in the output, whereas the Halton set shows a more stable performance (compare Figures \ref{Diff_40} and \ref{Diff_40_bad}).\\
 $ $\\
 \begin{figure}[htbp]
	\begin{minipage}{0.5\textwidth} 
	\includegraphics[width=\textwidth]{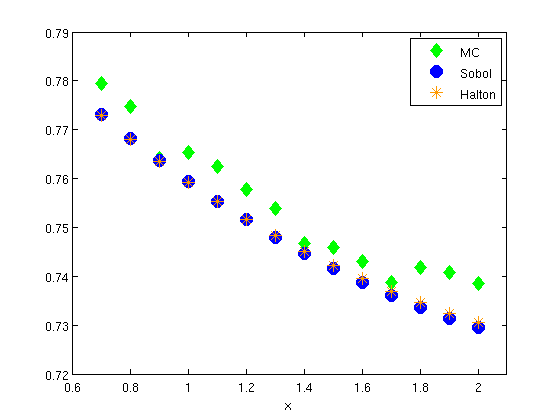}
	\caption{Influence of the starting value}
	\label{vary_x0}
	\end{minipage}
	\hfill
	\begin{minipage}{0.5\textwidth}
	For Figure \ref{vary_x0} we evaluated $k=40$ iterations of the algorithm with $N=30000$ (Q)MC points for different starting values $x$, ranging from $0.7$ to $2$. As expected, the discounted time of ruin decreases for increasing $x$.
	\end{minipage}
\end{figure}
\subsection{The deficit at ruin}
If we set $w(y,z):=y$, and $\delta=0$, we have $V(x)=\mathbb{E}(|X_{\tau-}|)$, the expected deficit at ruin. We use the same premium rate $c(\cdot)$ as before and again choose exponential distributions for the inter-arrival times and claim sizes with parameters $\lambda$ and $\mu$ respectively, since also in this case the true value $\mathbb{E}(|X_{\tau-}|)=\frac1\mu$ (for a classically reflected process) can be found in \cite{lin2003classical}. Figures \ref{DaR_20} and \ref{DaR_100} show the results for $k=20$ and $k=100$ iterations respectively. The reference value is again shown as a red line, in our case at $1.25$. The MC points are drawn in green, the Sobol points blue and the Halton points in orange. Table \ref{w2_20_table} and Table \ref{w2_100_table} contain the precise values along with the corresponding parameters.\\
$ $\\
Note again the difference between Figure \ref{DaR_20} and Figure \ref{DaR_100}, resulting from a different number of iterations $k$.
\begin{figure}[htbp]
	\begin{minipage}{0.5\textwidth} 
	\includegraphics[width=\textwidth]{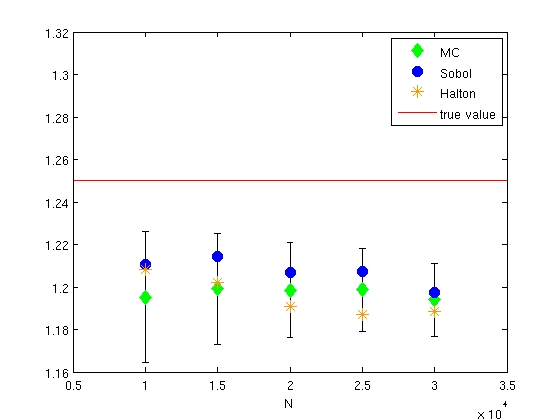}
	\caption{\footnotesize$k=20$ iterations of the algorithm.\normalsize\label{DaR_20}}
	\end{minipage}
	\hfill
	\begin{minipage}{0.5\textwidth}
	\includegraphics[width=\textwidth]{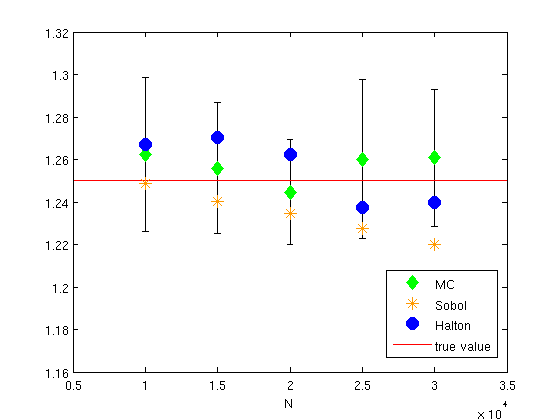}
	\caption{\footnotesize$k=100$ iterations of the algorithm.\normalsize\label{DaR_100}}
	\end{minipage}
	\vskip 1cm
	\begin{minipage}{0.45\textwidth}
	\begin{adjustbox}{width=\linewidth}
	 \renewcommand{\arraystretch}{1.2}
	\begin{tabular}{rrrrrr}
	  $x$ & $\lambda$ & $\mu$ & $w(y,z)$ & $k$ & $\delta$ \\
	  $1.2$ & $1$ & $0.8$ & $y$ & $20$ & $0$ \\
	  \hline $N$: & $10000$ & $15000$ & $20000$ & $25000$ & $30000$ \\
	  MC: & $1.1952$ & $1.1991$ & $1.1986$ & $1.1987$ & $1.1939$\\
	  Sobol: & $1.2105$ & $1.2142$ & $1.2070$ & $1.2074$ & $1.1975$ \\
	  Halton: & $1.2084$ & $1.2019$ & $1.1906$ & $1.1872$ & $1.1885$
	\end{tabular}}
	\end{adjustbox}\captionof{table}{\label{w2_20_table}}
	\end{minipage}
	\hfill
	\begin{minipage}{0.45\textwidth}
	\begin{adjustbox}{width=\linewidth}
	\renewcommand{\arraystretch}{1.2}
	\begin{tabular}{rrrrrr}
	  $x$ & $\lambda$ & $\mu$ & $w(y,z)$ & $k$ & $\delta$ \\
	  $1.2$ & $1$ & $0.8$ & $y$ & $100$ & $0$ \\
	  \hline $N$: & $10000$ & $15000$ & $20000$ & $25000$ & $30000$\\
	  MC: & $1.2624$ & $1.2558$ & $1.2446$ & $1.2602$ & $1.2607$\\
	  Sobol: & $1.2669$ & $1.2704$ & $1.2624$ & $1.2373$ & $1.2398$\\
	  Halton: & $1.2487$ & $1.2404$ & $1.2344$ & $1.2273$ & $1.2199$
	\end{tabular}
	\end{adjustbox}\captionof{table}{\label{w2_100_table}}
	\end{minipage}
\end{figure}
\begin{figure}[htbp]
	\begin{minipage}{0.5\textwidth} 
	Again, we plotted the absolute error for $k=40$ iterations of the algorithm and a varying number of (Q)MC points $N$. Figure \ref{Diff_DaR} shows the results using the same colorings as before.
	\end{minipage}
	\hfill
	\begin{minipage}{0.5\textwidth}
	\includegraphics[width=\textwidth]{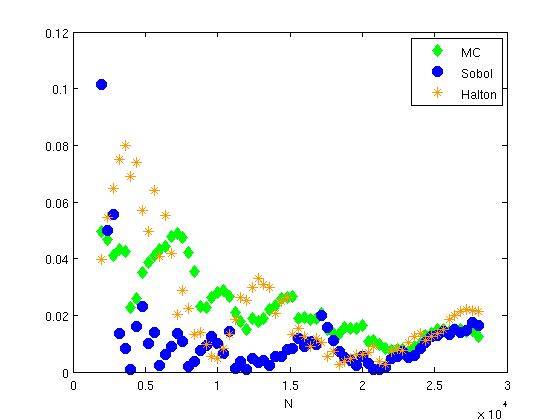}
	\caption{The absolute error for the deficit at ruin}
	\label{Diff_DaR} 
	\end{minipage}
\end{figure}

\begin{remark}
We considered in our numerical examples two test cases for which explicit (approximate) reference values are available. Certainly our approach is not restricted to this particular choice of model ingredients - which are $f_Y$, $f_W$ and the penalty function $w$.
\end{remark}


%

{\footnotesize
\bibliographystyle{abbrv}

\begin{thebibliography}{10}

\bibitem{AistDick2015}
C.~Aistleitner and J.~Dick.
\newblock Functions of bounded variation, signed measures, and a general
  {K}oksma-{H}lawka inequality.
\newblock {\em Acta Arith.}, 167(2):143--171, 2015.

\bibitem{AistPausSvanTich2016}
C.~Aistleitner, F.~Pausinger, A.~M. Svane, and R.~F. Tichy.
\newblock On functions of bounded variation.
\newblock {\em Math. Proc. Cambridge Philos. Soc.}, pages 1--15, to appear.

\bibitem{AlbKain02}
H.~Albrecher and R.~Kainhofer.
\newblock Risk theory with a nonlinear dividend barrier.
\newblock {\em Computing}, 68(4):289--311, 2002.

\bibitem{AsAlb2010}
S.~Asmussen and H.~Albrecher.
\newblock {\em Ruin probabilities}.
\newblock World Scientific, River Edge, 2nd edition, 2010.

\bibitem{atkinson1967numerical}
K.~E. Atkinson.
\newblock The numerical solution of fredholm integral equations of the second
  kind.
\newblock {\em SIAM Journal on Numerical Analysis}, 4(3):337--348, 1967.

\bibitem{Brandolini2013b}
L.~Brandolini, L.~Colzani, G.~Gigante, and G.~Travaglini.
\newblock A koksma--hlawka inequality for simplices.
\newblock In {\em Trends in Harmonic Analysis}, pages 33--46. Springer, 2013.

\bibitem{Brandolini2013a}
L.~Brandolini, L.~Colzani, G.~Gigante, and G.~Travaglini.
\newblock On the koksma--hlawka inequality.
\newblock {\em Journal of Complexity}, 29(2):158 -- 172, 2013.

\bibitem{DavisMarkovOpt93}
M.~H.~A. Davis.
\newblock {\em Markov models and optimization}.
\newblock Chapman \& Hall, London, 1993.

\bibitem{dick2007lattice}
J.~Dick, P.~Kritzer, F.~Y. Kuo, and I.~H. Sloan.
\newblock Lattice-nystr{\"o}m method for fredholm integral equations of the
  second kind with convolution type kernels.
\newblock {\em Journal of Complexity}, 23(4):752--772, 2007.

\bibitem{drmota1997sequences}
M.~Drmota and R.~F. Tichy.
\newblock {\em Sequences, discrepancies and applications}.
\newblock 1997.

\bibitem{edelsbrunner2016approximation}
H.~Edelsbrunner and F.~Pausinger.
\newblock Approximation and convergence of the intrinsic volume.
\newblock {\em Advances in Mathematics}, 287:674--703, 2016.

\bibitem{GerShiuTOR98}
H.~U. Gerber and E.~S.~W. Shiu.
\newblock On the time value of ruin.
\newblock {\em N. Am. Actuar. J.}, 2(1):48--78, 1998.

\bibitem{Gerber_Shiu05}
H.~U. Gerber and E.~S.~W. Shiu.
\newblock The time value of ruin in a {S}parre {A}ndersen model.
\newblock {\em N. Am. Actuar. J.}, 9(2):49--84, 2005.

\bibitem{gotz2002discrepancy}
M.~G{\"o}tz.
\newblock Discrepancy and the error in integration.
\newblock {\em Monatshefte f{\"u}r Mathematik}, 136(2):99--121, 2002.

\bibitem{Harman2010}
G.~Harman.
\newblock Variations on the {K}oksma-{H}lawka inequality.
\newblock {\em Unif. Distrib. Theory}, 5(1):65--78, 2010.

\bibitem{hlawka1961funktionen}
E.~Hlawka.
\newblock Funktionen von beschr{\"a}nkter variation in der theorie der
  gleichverteilung.
\newblock {\em Annali di Matematica Pura ed Applicata}, 54(1):325--333, 1961.

\bibitem{hua2012applications}
L.~K. Hua and Y.~Wang.
\newblock {\em Applications of number theory to numerical analysis}.
\newblock Springer-Verlag, Berlin-New York; Kexue Chubanshe (Science Press),
  Beijing, 1981.
\newblock Translated from the Chinese.

\bibitem{ikebe1972galerkin}
Y.~Ikebe.
\newblock The galerkin method for the numerical solution of fredholm integral
  equations of the second kind.
\newblock {\em Siam Review}, 14(3):465--491, 1972.

\bibitem{kuipers2012uniform}
L.~Kuipers and H.~Niederreiter.
\newblock {\em Uniform distribution of sequences}.
\newblock Courier Corporation, 2012.

\bibitem{kuo2003component}
F.~Y. Kuo.
\newblock Component-by-component constructions achieve the optimal rate of
  convergence for multivariate integration in weighted korobov and sobolev
  spaces.
\newblock {\em Journal of Complexity}, 19(3):301--320, 2003.

\bibitem{LeobacherPillichshammer2014}
G.~Leobacher and F.~Pillichshammer.
\newblock {\em Introduction to quasi-{M}onte {C}arlo integration and
  applications}.
\newblock Compact Textbook in Mathematics. Birkh\"auser/Springer, Cham, 2014.

\bibitem{Lundberg1903}
F.~Lundberg.
\newblock {\em Approximerad framst\"allning av sannolikhetsfunktionen.
  Aterf\"ors\"akring av kollektivrisker}.
\newblock Akad. Afhandling. Almqvist o. Wiksell, Uppsala, 1903.

\bibitem{novak2010tractability}
E.~Novak and H.~Wo{\'z}niakowski.
\newblock {\em Tractability of Multivariate Problems: Standard information for
  functionals}, volume~12.
\newblock European Mathematical Society, 2010.

\bibitem{Owen2005}
A.~B. Owen.
\newblock Multidimensional variation for quasi-{M}onte {C}arlo.
\newblock In {\em Contemporary multivariate analysis and design of
  experiments}, volume~2 of {\em Ser. Biostat.}, pages 49--74. World Sci.
  Publ., Hackensack, NJ, 2005.

\bibitem{PausingerSvane2015}
F.~Pausinger and A.~M. Svane.
\newblock A {K}oksma-{H}lawka inequality for general discrepancy systems.
\newblock {\em J. Complexity}, 31(6):773--797, 2015.

\bibitem{sloan1988quadrature}
I.~H. Sloan.
\newblock A quadrature-based approach to improving the collocation method.
\newblock {\em Numerische Mathematik}, 54(1):41--56, 1988.

\bibitem{sloan1989representation}
I.~H. Sloan and J.~N. Lyness.
\newblock The representation of lattice quadrature rules as multiple sums.
\newblock {\em Mathematics of computation}, 52(185):81--94, 1989.

\bibitem{sloan2001tractability}
I.~H. Sloan and H.~Wo{\'z}niakowski.
\newblock Tractability of multivariate integration for weighted korobov
  classes.
\newblock {\em Journal of Complexity}, 17(4):697--721, 2001.

\bibitem{StoerBulirsch2000}
J.~Stoer and R.~Bulirsch.
\newblock {\em Numerische {M}athematik. 2}.
\newblock Springer-Lehrbuch. [Springer Textbook]. Springer-Verlag, Berlin,
  fourth edition, 2000.
\newblock Eine Einf{\"u}hrung---unter Ber{\"u}cksichtigung von Vorlesungen von
  F. L. Bauer. [An introduction, with reference to lectures by F. L. Bauer].

\bibitem{Tichy1984_Integral}
R.~F. Tichy.
\newblock \"{U}ber eine zahlentheoretische {M}ethode zur numerischen
  {I}ntegration und zur {B}ehandlung von {I}ntegralgleichungen.
\newblock {\em \"Osterreich. Akad. Wiss. Math.-Natur. Kl. Sitzungsber. II},
  193(4-7):329--358, 1984.

\bibitem{twomey1963numerical}
S.~Twomey.
\newblock On the numerical solution of fredholm integral equations of the first
  kind by the inversion of the linear system produced by quadrature.
\newblock {\em Journal of the ACM (JACM)}, 10(1):97--101, 1963.

\bibitem{brunner1984iterated}
Brunner, H.: Iterated collocation methods and their discretizations for
  {V}olterra integral equations.
\newblock SIAM J. Numer. Anal. \textbf{21}(6), 1132--1145 (1984)

\bibitem{brunner1992implicitly}
Brunner, H.: Implicitly linear collocation methods for nonlinear {V}olterra
  equations.
\newblock Appl. Numer. Math. \textbf{9}(3), 235--247 (1992)

\bibitem{kumar1987new}
Kumar, S., Sloan, I.H.: A new collocation-type method for {H}ammerstein
  integral equations.
\newblock Math. Comp. pp. 585--593 (1987)

\bibitem{lin2003classical}
Lin, X.S., Willmot, G.E., Drekic, S.: The classical risk model with a constant
  dividend barrier: analysis of the {G}erber--{S}hiu discounted penalty
  function.
\newblock Insurance Math. Econom. \textbf{33}(3), 551--566 (2003)

\bibitem{LectureODE}
Grigorian, A.: Ordinary differential equation.
\newblock Lecture notes, available at
  {https://www.math.uni-bielefeld.de/\textasciitilde grigor/odelec2009.pdf} (2009)

\end{thebibliography}
}
\Addresses

\end{document}